\documentclass[12pt,leqno,a4paper]{amsart}
\usepackage{amssymb}  
\usepackage{amsmath}  
\usepackage{amsthm}  
\usepackage{verbatim} 

\theoremstyle{plain}
\newtheorem{theorem}{THEOREM}[section]
\newtheorem{proposition}{PROPOSITION}[section]
\newtheorem{lemma}{LEMMA}[section]
\newtheorem{corollary}{COROLLARY}[section]

\theoremstyle{definition}
\newtheorem{definition}{DEFINITION}[section]

\newtheorem{remark}{Remark}[section]

\numberwithin{equation}{section}
\newcommand{\ga}{\alpha}

\newcommand{\gd}{\delta}
\newcommand{\gD}{\Delta}

\newcommand{\gl}{\lambda}

\newcommand{\gm}{\mu}
\newcommand{\gn}{\nu}

\newcommand{\gs}{\sigma}

\newcommand{\go}{\omega}
\newcommand{\gO}{\Omega}
\newcommand{\ov}{\overline}
\newcommand{\de}{\partial}

\newcommand{\R}{\mathbb {R}}
\newcommand{\RN}{\R ^N}
\newcommand{\Rm}{\R ^m}

\newcommand{\ldue}{\textbf{L}^2}
\newcommand{\huno}{\textbf{H}_0^1}

\newcommand{\be}{\begin{equation}}
\newcommand{\ee}{\end{equation}}

\begin{document}

\title[Symmetry for cooperative elliptic systems]
{A symmetry result for semilinear cooperative elliptic systems}
\author[Damascelli]{Lucio Damascelli}
\address{ Dipartimento di Matematica, Universit\`a  di Roma 
" Tor Vergata " - Via della Ricerca Scientifica 1 - 00173  Roma - Italy.} 
\email{damascel@mat.uniroma2.it}
\author[Gladiali]{Francesca  Gladiali}
\address{Matematica e Fisica, Universit\`a  di Sassari  - Via Piandanna 4, 07100 Sassari - Italy.} 
\email{fgladiali@uniss.it}
\author[Pacella]{Filomena Pacella}
\address{Dipartimento di Matematica, Universit\`a di Roma
" La Sapienza " -  P.le A. Moro 2 - 00185 Roma - Italy.}
\email{pacella@mat.uniroma1.it}
\date{}
\thanks{Supported by PRIN-2009-WRJ3W7 grant}
\begin{abstract} In this paper we prove symmetry results for classical solutions of nonlinear cooperative elliptic systems in a ball or in annulus in $\RN$, $N \geq 2 $.
More precisely we prove that solutions having Morse index $j \leq N-1 $ are foliated Schwarz symmetric if the nonlinearity has a convex derivative and a full coupling condition is satisfied along the solution.
   \end{abstract}

\maketitle



\section{  \textbf{Introduction and statement of the results} }
\label{se:1}

We consider a semilinear elliptic system of the type
\begin{equation} \label{modprob} 
\begin{cases} - \gD U =F(|x|,U) \quad &\text{in } \gO \\
U= 0 \quad & \text{on } \de  \gO
\end{cases}
\end{equation}
where $F=(f_1 , \dots , f_m)$ is a function belonging to $C^{1, \ga }([0,+\infty )\times \Rm; \Rm)$, $\gO$ is a bounded domain in $\RN$ and $U=(u_1, \dots ,u_m)$ is a vector valued function in $\gO$,  $m, N \geq 2$.\\ 
Systems of this type arise in many applications in different fields (see e.g. \cite{Me}, \cite{Mu}).\par
In this paper we continue the study of the symmetry of solutions of \eqref{modprob} in rotationally symmetric domains, namely when $\Omega $ is a ball or an annulus, started in \cite{DamPac}.\par
We recall that if $\Omega $ is a ball and the system is cooperative then in  \cite{BS1}, \cite{deF1}, \cite{deF2}, \cite{Tr} it is proved that every positive solution $U$ (i.e. $u_i >0$ for any $i=1, \dots m $) is radial and radially decreasing if every $f_i $ is non increasing with respect to $|x|$. This result is obtained through the famous ''moving plane method'' (\cite{S}), as in the scalar case (\cite{GNN}). However this method does not allow neither to consider sign changing solutions nor nonlinearities which are not radially nonincreasing with respect to $|x|$. Moreover, as it is a continuation method, it does not provide symmetry results in the annulus.\par 
Another symmetry result for cooperative systems in a ball, obtained by the symmetrization method,  is proved in \cite{KP}.\par
In the scalar case  a different approach was introduced in \cite{P} and later extended in \cite{PW} and \cite{GPW} (see also \cite {PR}) which allows to cover cases which cannot be treated by using the moving plane method. This approach is essentially based on Morse index considerations and on some convexity properties of the nonlinearity. The symmetry obtained for the solutions of \eqref {modprob} with this method is an axially symmetry which is indeed what one expects. \par
It is then natural to try to extend the method of \cite{P} to the case of systems like \eqref{modprob}. However this extension presents several difficulties, as we explain below, and cannot be done straightforwardly. In \cite{DamPac} we obtained symmetry results in this direction for solutions of \eqref {modprob} imposing some assumptions on the nonlinearity which in particular imply that it is convex in all variables. For some kind of systems (e.g. power type odd nonlinearities) this  hypothesis does not allow to get symmetry of sign changing solutions.\par
Here we will consider other type of nonlinearities, as in \cite{GPW}, which cover this case. Moreover in \cite{DamPac}, if $m \geq 3$ the nonlinearity should be the sum of functions depending only on two variables. For the nonlinearities we consider here we do not need this assumption.\par
To present our results and to be more precise we need some preliminary definitions.

 \begin{definition}\label{foliatedSS}
Let $\gO$ be a rotationally symmetric domain in $\RN$, $N\geq 2$. We say that a continuous vector valued function $U=(u_1 , \dots ,u_m): \gO \to \Rm$ is foliated Schwarz symmetric if each component $u_i$ is foliated Schwarz symmetric with respect to the same vector $p \in \RN$. In other words there exists a vector $p \in \RN$, 
$|p|=1$, such that $U(x)$ depends only on $r=|x|$ and $\theta= \arccos \left ( \frac {x}{|x|}\cdot p  \right ) $ and $U$ is (componentwise) nonincreasing in $\theta $.
\end{definition}
\begin{remark} Let us observe that if $U$ is a solution of \eqref{modprob} and the system satisfies some coupling conditions, as required in Theorem \ref{f'convessa}, then the foliated Schwarz symmetry of $U$ implies that either $U$ is radial or it is strictly decreasing in the angular variable $\theta $. This will be deduced by the proof of Theorem \ref{f'convessa}.
\end{remark}

\begin{definition}\label{MorseIndex}  Let $U$ be a $C^2 (\gO;\Rm)$ solution of \eqref{modprob}. 
\begin{itemize}
\item [i)]
We say that $U$ is linearized stable (or that it has zero Morse index) if the quadratic form 
\begin{equation} \label{quadraticformlineariz}
\begin{split}
Q_U (\Psi ;\gO  ) &= \int _{\gO} \left [ |\nabla \Psi |^2 - J_F (|x|, U)(\Psi ,\Psi ) \right ]dx= \\
& \int _{\gO}\left [   \sum _{i=1}^m |\nabla \psi _i |^2 -\sum _{i,j=1}^m \frac {\de f_i}{\de u_j}(|x|, U(x)) \psi _i \psi _j \right ]  \, dx \geq 0
\end{split}
\end{equation}
for any $\Psi = (\psi _1, \dots , \psi _m) \in C_c^1 (\gO;\Rm)$ where $J_F (x,U)$ is the jacobian matrix of $F$ computed at $U$.
\item [ii)] $U$ has (linearized) Morse index  equal to the integer $\mu=\mu (U)\geq 1$ if $\mu $ is the maximal dimension of a subspace of  $ C_c^1 (\gO;\Rm)$ where the quadratic form is negative definite.
\item [iii)] $U$ has infinite (linearized) Morse index  if for any integer $k$ there is a $k$-dimensional subspace of  $ C_c^1 (\gO;\Rm)$ where the quadratic form is negative definite.
\end{itemize}
\end{definition}

\begin{definition} 
\begin{itemize}
\item We say that the system \eqref{modprob} is cooperative or weakly coupled in an open set $\gO' \subseteq \gO $  if 
$$  \frac {\de f _i} {\de u _j} (|x|, u_1, \dots , u_m) \geq 0 \quad \forall \; (x,u_1,\dots ,u_m ) \in \gO '  \times  \Rm 
$$
for any  $i,j=1,\dots ,m$ with  $ i \neq j$.
  \item  We say that the system \eqref{modprob} is fully  coupled along a solution $U$ in an open set $\gO' \subseteq \gO $ if  it is cooperative in $\gO ' $ and 
 in addition $\forall I,J \subset \{1, \dots ,m \}$ such that $I \neq \emptyset $, $J \neq \emptyset $, $I \cap J = \emptyset $, $I \cup J = \{1, \dots ,m \} $ there exist $i_0 \in I$, $j_0 \in J $ such that 
$$\text{meas }(\{ x \in \gO ' :   \frac {\de f _{i_0}} {\de u _{j_0}} (|x|, U(x)) > 0 \}) >0$$  
 \end{itemize}
\end{definition}

Note that the previous definition means that the linearized system at a solution $U$ is weakly or fully coupled. \par
\smallskip 

 Let $e\in S^{N-1}$ be a direction, i.e. $e \in \RN $, $|e|=1$,  and let us define the set 
 $$   \gO (e)= \{ x \in \gO : x \cdot e >0\}  $$
 \smallskip 
 
 In \cite {DamPac} we proved the following result
 
 \begin{theorem}[ \cite {DamPac} ]  \label{fconvessa}   Let $\gO$ be a ball or an annulus in $\RN $, $N \geq 2$, and let $U \in C^{3,\ga }(\ov {\gO}; \Rm)$ be a solution of \eqref{modprob} with Morse index $\gm (U) \leq N $. 
Moreover assume that:
\begin{itemize} 
\item [i) ]  The system is fully coupled along $U$ in $ \gO (e) $, for any $e\in S^{N-1}$.
 \item [ii) ] For any $i,j=1, \dots m $  \  $\frac {\de f_i} {\de u_j}(|x|,u_1, \dots , u_m)$ is nondecreasing in each variable  $u_k$,  $k=1, \dots , m $, for any $|x| \in \gO $.
\item [iii)]    if $m \geq 3 $ then,   for any $i\in  \{1, \dots ,m \}$,   $f_i(|x|,u_1,\dots , u_m)= \sum _{k \neq i} g_{ik}(|x|,u_i,u_k)$ where 
$g_{ik} \in C^{1, \ga }([0,+\infty )\times \R ^2)$.
  \end{itemize}
 Then $U$ is  foliated Schwarz symmetric and  if $U$ is not radial then it is strictly decreasing in the angular variable (see Definition \ref{foliatedSS}).
\end{theorem} 

The monotonicity hypothesis ii) implies that each $f_i$ is convex with respect to each variable $ u_j $, $i,j= 1, \dots ,m$. \par 
Note that the regularity of the solution $U$ in Theorem \ref{fconvessa}, as in the next one, is a consequence of the regularity of the nonlinearity $F$ which is assumed to be of class $C^{1,\alpha}$. Moreover it implies that the derivatives of $U$ are classical solutions of the linearized system which is used in the proof.\par
Though the results of  Theorem \ref{fconvessa} also applies to sign changing solutions of \eqref{modprob} there are cases when the nonlinearity $F$, if considered on the whole $\R^m$ is not convex in the $u_k $ variables. This is, for example, the case of
$$F(u_1, u_2) =( f_1(u_1, u_2), f_2(u_1, u_2)  )  = ( | u_2 | ^{p-2} u_2 , | u_1 | ^{q-2} u_1) \quad, \quad p,q >1
$$

As announced we consider here other nonlinearities which cover this case and we also do not assume iii) of Theorem \ref{fconvessa} when $m \geq 3$. More precisely our symmetry results are the following.

\begin{theorem} \label{f'convessa} Let $\gO$ be a ball or an annulus in $\RN $, $N \geq 2$, and let $U \in C^{3,\ga }(\ov {\gO}; \Rm)$ be a solution of \eqref{modprob} with Morse index $\gm (U) \leq N-1 $. 
Moreover assume that:
\begin{itemize} 
\item [i) ]  The system is fully coupled along $U$ in $\gO $  
\item [ii) ] For any $i,j=1, \dots m $  the function   $\frac {\de f_i} {\de u_j}(|x|,S)$ is convex in $S=(s_1, \dots , s_m)$:
$$\frac {\de f_i} {\de u_j}(|x|,t  S'+(1-t) S'') \leq  t \frac {\de f_i} {\de u_j}(|x|, S')+(1-t)\frac {\de f_i} {\de u_j}(|x|, S'') $$ 
for any $t\in[0,1]$,  $S'$, $ S'' \in \R^m$  and $x \in \Omega $.
  \end{itemize}
 Then $U$ is  foliated Schwarz symmetric and  if $U$ is not radial then it is strictly decreasing in the angular variable (see Definition \ref{foliatedSS}).
\end{theorem} 

 \begin{remark} \label{ConfrontoAltroLavoro} As observed before the assumption ii) of Theorem \ref{f'convessa} allows to get the symmetry of solutions in cases not covered by Theorem \ref{fconvessa} and also to remove the assumption iii) for $m \geq 3$.\\
 On the contrary the assumption on the Morse index is more restrictive since we require $\gm (U) \leq N-1 $. Note that in the scalar case, in \cite {PW} (see also \cite{GPW}, \cite {PR}) we could get the symmetry result up to Morse index less than or equal to $N$. For systems some serious difficulty arises and we are not able to consider the case $\gm (U) = N $.
  \end{remark}

   \begin{theorem} \label{corollario1}
 Under the assumptions of Theorem \ref{f'convessa} assume that $U$ is a \emph{nonradial} solution of \eqref{modprob} and 
  either 
  \begin{itemize}  
   \item[a) ]  \  \  $U$ has Morse index one \par 
  \ \ \   \ \ \ \ or \par 
   \item[b) ]  \  \   there exist $ i_0, j_0 \in \{1, \dots ,m  \} $ such that the function $\frac {\de f_{i_0}} {\de u_{j_0}}(|x|,S)$ satisfies the following strict convexity assumption:  
  \be  \label{strictconvexity} \frac {\de f_{i_0}} {\de u_{j_0}}(|x|,t  S'+(1-t) S'') <  t \frac {\de f_{i_0}} {\de u_{j_0}}(|x|, S')+(1-t)\frac {\de f_{i_0}} {\de u_{j_0}}(|x|, S'') 
  \ee 
for any $t \in (0,1) $, whenever    $x \in \Omega $ and  $S'$, $ S'' \in \R^m$  satisfy $s'_k \neq s'' _k $ for any $k \in \{1, \dots ,m \}$. 
 \end{itemize}
  Then necessarily
\be \label{superfullycoupling1}
 \sum_{j=1}^m  \frac {\de f _i}{\de u _j}(r, U(r, \theta)) \frac {\de u_j } {\de \theta  }(r, \theta )=
 \sum_{j=1}^m  \frac {\de f _j}{\de u _i}(r, U(r, \theta))   \frac {\de u_j } {\de \theta  }(r, \theta )
\ee
 for any    $i=1,\dots ,m $, whith $(r, \theta)$ as in Definition \ref{foliatedSS}. \\ 
   In particular  if $m=2$ then \eqref{superfullycoupling1} implies that 
 \be \label{superfullycoupling2} 
 \frac {\de f _1}{\de u _2}(|x|, U(x))=  \frac {\de f _2}{\de u _1}(|x|, U(x)) \; , \quad \forall \, x \in \gO \;.
 \ee
 \end{theorem}
Note that \eqref{superfullycoupling1} and \eqref{superfullycoupling2} were also deduced in \cite{DamPac} under the assumptions of Theorem \ref{fconvessa} but only for Morse index one solutions.

\begin{remark}[Radial symmetry of stable solutions]    
  The symmetry result of Theorem \ref {f'convessa} holds in particular for stable solutions of \eqref{modprob}.\par
  However in this case it is easy to  get that the solution is radial without any assumption on the nonlinearity (see \cite {DamPac}, Theorem 1.2).
   \end{remark}
The general strategy to get symmetry results for systems, using the ideas of \cite{P} and \cite{PW} for the scalar case, is described in \cite{DamPac}.
Let us just point out here that one of the main difficulties encountered is that the linearized operator $L_U$ (see Section 3 for the definition) is not in general selfadjoint, due to the fact that the Jacobian matrix $J_F (|x|,U) $ is not symmetric. To bypass this difficulty, as in \cite {DamPac} we associate to $J_F (|x|,U) $ its symmetric part $\frac 12 \left (J_F (|x|,U) + J_F^t (|x|,U) \right ) $ where $J_F^t $ is the transpose of the matrix $J_F $. To this symmetric part is associated a selfadjoint operator $\tilde {L}_U $ whose spectrum can be variationally characterized. The crucial, simple remark in \cite{DamPac} is that the quadratic form associated to the linearized operator $L_U$  is the same as the quadratic form associated to $\tilde {L}_U $.\\
Obviously if the jacobian matrix $J_F (|x|,U) $  is symmetric  the two linear operators coincide. This happens, for example, when the system is  of  
\emph{gradient type}, i.e. when $F = \nabla  g $ for some scalar function $g$ (see \cite {deF2}), since, in this case, the linearized operator corresponds to the second derivative of a suitable associated functional.\par
However this is not the case for many interesting systems, like e.g. the so called  \emph{hamiltonian systems} (see \cite {deF2}).\par
We refer to \cite {DamPac} for further comments on this issue and on the analysis of the Morse index of solutions of \eqref{modprob}.\par
The outline of the paper is the following. In Section 2 we recall some  definitions and results about linear systems. In Section 3 we prove some preliminary results for  solutions of  \eqref{modprob} and prove  Theorem \ref {f'convessa} and Theorem \ref{corollario1}. Finally in Section 4 we present a few examples.

\section{  \textbf{Preliminaries on linear systems} }
\label{se:2}
In this section we recall several known results, almost all are proved  in the paper \cite{DamPac}. \\
Let $\gO$ be any smooth bounded domain in $\RN$, $N \geq 2$, and $D$  a $m \times m $ matrix with bounded entries:
\begin{equation}\label{ipotesiD} D= \left (  d_{ij} \right ) _{i,j=1}^m \; , \; d_{ij} \in L^{\infty} (\gO)
\end{equation}
Let us consider the linear elliptic  system
\begin{equation}  \label{linearsystem} 
\begin{cases} - \gD U + D(x) U = F \quad &\text{in } \gO \\
U= 0 \quad & \text{on } \de  \gO
\end{cases}
\end{equation}
i.e. 
$$
\begin{cases}
 - \gD u_1 +  d_{11} u_1+  \dots   +d_{1m} u_m =f_1  & \text{ in } \gO  \\
  \dots  \dots & \dots \\ 
 - \gD u_m +  d_{m1} u_1+  \dots   +d_{mm} u_m =f_m  & \text{ in } \gO \\
  u_1=  \dots  = u_m   =0  & \text {  on } \de \gO 
\end{cases}
$$
 where $F=(f_1, \dots , f_m) \in (L^2 (\gO))^m $,  $U=(u_1, \dots , u_m) $. \\
 This kind of linear system  appears in the linearization of the semilinear  elliptic system  \eqref{modprob}.
 \begin{definition}\label{sistemilineariaccoppiati} 
 The system \eqref{linearsystem} is said to be
 \begin{itemize}
 \item  \emph{cooperative} or \emph{weakly coupled} in $ \gO $ if 
 \begin{equation} \label{weaklycoupled} d_{ij}  \leq 0 \, \text{a.e. in }\gO , \quad \text{whenever } i \neq j
 \end{equation}
 \item \emph{fully coupled} in $ \gO $ if it is weakly coupled in $\gO $ and the following condition holds:
 \begin{equation}\label{fullycoupled}
 \begin{split}
  \forall \, I,J \subset \{ 1, \dots , m \}\, ,\, & I,J \neq \emptyset \, , \, I \cap J = \emptyset \, , \, I \cup J = \{ 1, \dots , m \} \, \\
  &  \exists i_0 \in I \, , \, j_0 \in J \, : \text{meas } (\{ x \in \gO  : d_{i_0j_0} <0 \}) >0
  \end{split}
 \end{equation}
 \end{itemize}
 \end{definition}
It is well known that either condition \eqref{weaklycoupled} or conditions \eqref{weaklycoupled} and \eqref{fullycoupled} together are needed in the proofs of maximum principles for systems (see \cite{deF1}, \cite{deFM}, \cite{Si} and the references therein).
In particular if both are fulfilled the strong maximum principle holds as it is shown in the next  theorem.\par
 \textbf{Notation remark}:  here and in the sequel inequalities involving vectors should be understood  to hold componentwise, e.g. if  $\Psi = (\psi _1 , \dots , \psi _m) $,   $\Psi $ nonnegative  means that $\psi _j \geq 0 $ for any index $j=1, \dots , m$. 
 
\begin{theorem}\label{SMP} (Strong Maximum Principle and Hopf's Lemma).  Suppose that \eqref{ipotesiD},   \eqref{weaklycoupled} and \eqref{fullycoupled} hold  and $U=(u_1, \dots , u_m) \in C^1 (\ov {\gO};\Rm)$  is a weak solution of the inequality 
$$- \gD U + D(x) U \geq 0 \text{ in } \gO
$$ 
  i.e. 
\begin{equation} 
 \int _{\gO}  \nabla U \cdot  \nabla \Psi   + D(x) (U ,\Psi )= 
 \int _{\gO}\left [   \sum _{i=1}^m  \nabla u_i \cdot  \nabla \psi _i  +\sum _{i,j=1}^m d_{ij}(x) u_i  \psi _j \right ]  \, dx \geq 0
\end{equation}
for any nonnegative $\Psi = (\psi _1 , \dots , \psi _m) \in C_c^1 (\gO  ; \Rm)$.  \\
If $U \geq  0 $ in $\gO $, then either $U \equiv 0 $ in $\gO $ or $U>0 $ in $\gO $. In the latter case if $P \in \partial \gO $ and $U(P)=0$ then $\frac {\de U }{\de \gn }(P) < 0$, where $\gn $ is the unit exterior normal vector at $P$.
\end{theorem}

We are interested in the quadratic form associated to system \eqref{linearsystem}, namely
\begin{equation} \label{formaquadraticalineari}
\begin{split}
Q (\Psi ; \gO ) &= \int _{\gO} \left [ |\nabla \Psi |^2 + D(x) (\Psi ,\Psi )\right ] dx= \\
& \int _{\gO}\left [   \sum _{i=1}^m |\nabla \Psi _i |^2 +\sum _{i,j=1}^m  d_{ij}(x)  \Psi _i \Psi _j \right ]  \, dx 
\end{split}
\end{equation}
for  $\Psi \in C_c^1 (\gO;\Rm)$ ( or $\Psi \in  H_0^1 (\gO;\Rm)$ ).\\
It is easy to see that this quadratic form coincides with the quadratic form  associated to the symmetric system
\begin{equation}  \label{symmetricsystem}  
\begin{cases} - \gD U + C(x) U = F \quad &\text{in } \gO \\
U= 0 \quad & \text{on } \de  \gO
\end{cases}
\end{equation}
i.e. 
$$
\begin{cases}
 - \gD u_1 +  c_{11} u_1+  \dots   +c_{1m} u_m &=f_1 \\
  \dots  \dots & \dots \\ 
 - \gD u_m +  c_{m1} u_1+  \dots   +c_{mm} u_m &=f_m
\end{cases}
$$
 where 
 \begin{equation}\label{matricesimmetricaassociata} C=  \frac 12 (D + D^{t}) \quad  \text { i.e. } \quad C=(c_{ij}), \quad     c_{ij} =  \frac 12 \, (d_{ij} + d_{ji})
 \end{equation}
So to study the sign of the quadratic form $Q$ we can also use the properties of the symmetric system \eqref {symmetricsystem}.\\
Therefore we review briefly the spectral theory for this kind of simmetric systems, and use it to prove some results that we need for the possible nonsymmetric system  \eqref{linearsystem}.
\begin{remark} If system  \eqref{linearsystem} is cooperative, respectively fully coupled, so is  the associate symmetric system  \eqref{symmetricsystem}. 
\end{remark}
 
 Let $\gO$ be a bounded domain in $\RN$, $N \geq 2$, and consider for $m \geq 1$ the Hilbert spaces 
 $\ldue = \textbf{L}^2 (\gO ) = \left ( L^2(\gO) \right )^m$, $\huno = \textbf{H}_0^1 (\gO)= \left ( H_0^1(\gO) \right )^m$, where if $f=(f_1, \dots , f_m)$, $g=(g_1, \dots , g_m)$ the scalar products are defined by 
 \be 
 \begin{split}
 &(f,g)_{\ldue}= \sum_{i=1}^m (f_i ,g_i )_{L^2(\gO)}=  \sum_{i=1}^m \int _{\gO} f_i \, g_i  \, dx \\ 
 &(f,g)_{\huno}= \sum_{i=1}^m (f_i ,g_i )_{H_0^1(\gO)}=  \sum_{i=1}^m \int _{\gO} \nabla f_i \,\cdot \, \nabla g_i  \, dx
\end{split}
 \ee
Let $C=C(x)=(c_{ij}(x))_{i,j=1}^m $ a symmetric matrix whose elements are bounded functions:
\be \label{ipotesiC} c_{ij} \in L^{\infty} \quad , \quad c_{ij}=c_{ji} \quad \text{ a.e. in }  \gO
\ee  
 and consider the bilinear form
 \be \label{bilinearform} B(U,\Phi)= \int _{\gO} \left [ \nabla U \cdot \nabla \Phi + C(U, \Phi) \right ] = \int _{\gO} \left [\sum_{i=1}^m  \nabla u_i \cdot \nabla \phi _i +\sum _{i,j=1}^m c_{ij}u_i \phi _j \right ] 
 \ee
 As recalled and discussed in  \cite {DamPac}, by the spectral theory of selfadjoint operators  there exist a sequence $\{\gl _j \}$ of eigenvalues, with
 $- \infty < \gl _1 \leq \gl _2 \leq \dots $, $\lim _{j \to + \infty} \gl _j = + \infty $, and a corresponding sequence of eigenfunctions $\{ W^j \}$, 
 $W^j \in \huno \cap  C^{1}(\gO;\Rm) $ that weakly solve the systems
 \be \label{sistemaautovalori}
 \begin{cases} - \gD W^j + C  W^j =   \gl _j W^j &  \text{ in }  \gO \\
  W^j =0 & \text{ on } \de \gO
 \end{cases}
 \ee
i.e. if $W^j=(w_1, \dots , w_m)$
$$
\begin{cases}
 - \gD w_1 +  c_{11} w_1+  \dots   +c_{1m} w_m &= \gl _j  w_1 \\
  \dots  \dots & \dots \\ 
 - \gD w_m +  c_{m1} w_1+  \dots   +c_{mm} w_m &=\gl _j  w_m
\end{cases}
$$ 
 that satisfy the following properties.
 In what follows if $\gO '$ is a subdomain of $\gO $ we denote by $\gl _k (\gO ')$ the eigenvalues of the same system with $\gO $ substituted by $\gO '$.
 
\begin{proposition}\label{varformautov} Suppose that $C=(c_{ij})_{i,j=1}^m $ satisfies \eqref{ipotesiC}, and let $\{\gl _j \}$, $\{ W^j \}$ be the sequences of eigenvalues and eigenfunctions that satisfy \eqref{sistemaautovalori}.\\
Define the Rayleigh quotient
\be R(V)= \frac {B(V,V)}{(V,V)_{\ldue}} \quad \text{ for } V \in \huno \quad V \neq 0
\ee
with $B(.,.)$ as in \eqref{bilinearform}.
Then the following properties hold, where  $\textbf{V}_k$ denotes a $k$-dimensional subspace of $\huno$ and the orthogonality conditions  $V \bot W_k $  or  $V \bot \textbf{V}_k $  stand for the orthogonality in $\ldue$.  
\begin{itemize}
\item[i)] $\gl _1 = \text{ min } _{V\in \huno \,, \, V \neq 0} R(V) =  \text{ min } _{V\in \huno \, , \,  (V,V)_{\ldue}=1 } B(V,V) $
\item[ii)]   $\gl _m = \text{ min } _{V\in \huno \,, \, V \neq 0 \, , \, V \bot W_1,\dots , V \bot W_{m-1}} R(V)$\\ $ =  \text{ min } _{V\in \huno \, , \,  (V,V)_{\ldue}=1 \, , \, V \bot W_1,\dots , V \bot W_{m-1} } B(V,V) $  if $m \geq 2$ 
\item[iii)] $\gl _m = \text{ min } _{\textbf{V}_m }   \text{ max } _{ V \in \textbf{V}_m \, , \, V \neq 0 } R(V) $
\item[iv)] $\gl _m = \text{ max } _{\textbf{V}_{m-1} }   \text{ min } _{ V \bot  \in \textbf{V}_{m-1 }\, , \, V \neq 0 } R(V) $
\item[v)] If $W \in \huno$, $W \neq 0$, and $R(W)= \gl _1$, then $W$ is an eigenfunction corresponding to $\gl _1$.
\item[vi)] $ \lim _{\text{ meas }(\gO ') \to 0} \gl _1 (\gO ') = + \infty $
\item[vii)] If the system is cooperative in $\gO $ and $W $ is a first eigenfunction, then $W^{+}$ and $W^{-}$ are eigenfunctions, if they do not vanish.
\item[viii)] If the system is fully coupled in $\gO $,  then the first eigenfunction does not change sign in $\gO $ and  the first eigenvalue is simple, i.e. up to scalar multiplication there is only one eigenfunction corresponding to the first eigenvalue. 
\item[ix)] Assume that the system is fully coupled in $\gO $,  $C'  = \left (  c'_{ij} \right )_{i,j=1}^m $ is another matrix that satisfies \eqref{ipotesiC}, and let $\{{ \gl '} _k  \}$, 
 be the sequence of eigenvalues  of  the corresponding system.
If $c_{ij} \geq  {c'} _{ij}  $ for any $i,j=1, \dots ,m$ then  $ \gl _1 \geq { \gl '} _1$.
\end{itemize}
\end{proposition}
 
 \smallskip
  Let us turn back to the (possibly) nonsymmetric  cooperative system  \eqref{linearsystem} with the matrix  $D=(d_{ij})_{i,j=1}^m$  satisfying 
\be \label{ipotesiDcooperativa}
d_{ij} \in L^{\infty} (\gO) \quad, \quad   d_{ij} \leq 0 \, , \quad \text{whenever } i \neq j 
\ee

In the sequel we shall indicate by $ \gl _j^{\text{(s)}}=  \gl _j^{\text{(s)}}(-\gD +D; \gO )    $ the  eigenvalues of the associated symmetric system
 \eqref{symmetricsystem}. 
  Analogously the corresponding eigenfunctions will be indicated by $W_j^{\text(s)}$.\\
 We also  denote the bilinear form associated with the symmetric system \eqref{symmetricsystem} by
   $$ B^{\text{s}}(U,\Phi)= \int _{\gO} \left [ \nabla U \cdot \nabla \Phi + C(U, \Phi) \right ] = \int _{\gO} \left [\sum_{i=1}^m  \nabla u_i \cdot \nabla \phi _i +\sum _{i,j=1}^m c_{ij}u_i \phi _j \right ]  
 $$
  As already remarked, the quadratic form \eqref{formaquadraticalineari} associated to the system \eqref{linearsystem} coincides with that associated to  system \eqref{symmetricsystem}, i.e. 
$$ Q (\Psi ; \gO  ) = \int _{\gO} |\nabla \Psi |^2 + D(x) (\Psi ,\Psi ) = B^{\text{s}}(\Psi,\Psi)
$$ 
  if   $\Psi \in  H_0^1 (\gO;\Rm)$. \par 
\medskip
\begin{definition} We say that the maximum principle holds for the operator $- \gD + D$ in an open set  $\gO ' \subseteq \gO $ if  any $U \in \textbf{H}^1 (\gO ')$ such that 
\begin{itemize}
\item $U \leq  0 $ on $ \de \gO ' $ (i.e. $U^{+} \in \textbf{H}_0^1 (\gO ')$) 
\item  $- \gD U + D(x)U \leq 0 $ in $\gO ' $ (i.e. 
$\int \nabla U \cdot  \nabla \Phi + D(x) (U, \Phi ) \leq 0 $
for any nonnnegative  $\Phi \in \textbf{H}_0^1 (\gO ') $)
\end{itemize}
 satisfies $U \leq 0 $ a.e. in $\gO $.
\end{definition}
Let us denote by $  \gl _j^{\text{(s)}}  (\gO ') >0  $ the sequence of the eigenvalues of the symmetric system in an open set $\gO ' \subseteq \gO$.
\begin{theorem}\label{WMP} [Sufficient condition for weak maximum principle] Under the hypothesis \eqref{ipotesiDcooperativa}, if
 $  \gl _1^{\text{(s)}}  (\gO ') >0  $ then the maximum principle holds for $- \gD + D$ in $\gO ' \subseteq \gO $. \\
\end{theorem}
  
 Almost immediate consequences of the previous theorem are the  following ''Classical'' and ''Small measure'' forms of the weak maximum principle (see   \cite{BS2}, \cite{deFM},  \cite{PrW}, \cite {Si}).

\begin{corollary}  
\begin{itemize}
\item[i) ] If \eqref{weaklycoupled} holds and  $D$ is a.e. nonnegative definite in $\gO ' $ then the maximum principle holds for
$- \gD + D$ in $ \gO ' $. 
\item[ii) ]
There exists $\gd >0 $, depending on $D$, such that for any subdomain $\gO ' \subseteq \gO $ the maximum principle holds for $- \gD + D$ in $\gO ' \subseteq \gO $ provided $|\gO ' | \leq \gd$. 
\end{itemize}
\end{corollary}
  
Obviously the converse of Theorem \ref{WMP} holds if $D=C$ is symmetric: if the maximum principle holds for $- \gD + C$ in $\gO '$ then $\gl _1^{\text{(s)}} (\gO ' )>0$.
In fact if $\gl _1^{\text{(s)}} (\gO ' )\leq 0 $ since the system is cooperative (and symmetric) there exists a nontrivial nonnegative first eigenfunction $\Phi _1 \geq 0 $, 
$\Phi \not \equiv 0$, and  the maximum principle does not hold, since $- \gD \Phi _1 + C\, \Phi _1 = \gl _1 \Phi _1 \leq 0 $ in $\gO '$, $\Phi _1 =0 $ on $\de \gO '$, 
while  $\Phi _1 \geq 0 $ and
$\Phi _1 \neq 0$.\\
However this is not true for general nonsymmetric systems. Roughly speaking the reason is that there is an equivalence between the validity of the maximum principle for the operator $- \Delta + D$ and the positivity of its  principal eigenvalue $\tilde {\gl _1}$, whose definition  is given below, and  the inequality 
$\tilde {\gl _1} (\gO ') \geq    \gl _1^{\text{(s)}}  (\gO ' )  
$,  which can be  strict, holds. \par  
\medskip

More precisely we recall  that the \emph{principal eigenvalue}  of the operator \\
 $- \gD + D$ in an open set $\gO ' \subseteq \gO $ is defined as  
\be
\begin{split}
 \tilde {\gl _1} (\gO  ') &= 
   \sup \{ \gl \in \R : \exists \, \Psi \in W^{2,N}_{loc} (\gO ' ; \Rm) \; \text { s.t. } \\ 
&      \Psi  >0   - \gD \Psi + D(x) \Psi - \gl \Psi \geq 0 \text{ in } \gO '\}
\end{split}
\ee
(see \cite{BS2} and the references therein, and also \cite{BNV} for the case of scalar equations).\\
We then have:
\begin{proposition}\label{principaleigenvalue} Suppose that the system \eqref{linearsystem} is fully coupled in an open set  $\gO ' \subseteq \gO $. Then:
\begin{itemize}
\item[i) ] there exists a positive eigenfunction $\Psi _1 \in  W^{2,N}_{loc} (\gO ' ; \Rm)$ which satisfies  
\be \label{autofprincipale}   - \gD \Psi _1 + D(x) \Psi _1 = \tilde {\gl _1 } (\gO ') \Psi _1 \text { in } \gO\, , \quad \Psi _1 >0    \text { in }  \gO ' \, , \quad  \Psi _1 =0    \text { on } \de \gO ' 
\ee
Moreover the principal eigenvalue is simple, i.e. any function that satisfy \eqref{autofprincipale} must be a multiple of $\Psi _1$.
\item[ii) ]  the maximum principle holds for the operator $- \gD + D$ in $\gO '$ if and only if $\tilde {\gl _1} (\gO ' ) >0$
\item[iii) ] if there exists a positive function $\Psi \in  W^{2,N}_{loc} (\gO ' ; \Rm) $ such that   $\Psi  >0 , \,  - \gD \Psi + D(x) \Psi  \geq 0 $ in $\gO '$, then either 
$\tilde {\gl _1} (\gO  ') >0$ or $\tilde {\gl _1} (\gO ') =0 $ and $\Psi = c\, \Psi _1$ for some constant $c$.
\item[iv) ]  $\tilde {\gl _1} (\gO ' ) \geq  \gl _1^{\text{(s)}}  (\gO ' ) $, with equality if and only if $\Psi _1 $ is also the first eigenfunction of the symmetric operator $- \gD + C$ in $\gO '$ ,  $ C=\frac 12 (D + D^{\text{t}})$.
If  this is the case the equality $C(x) \Psi _1 =D(x) \Psi _1$ holds and,  if $m=2$, this implies that $d_{12}= d_{21}$.
\end{itemize}
\end{proposition}
 
 \section{  \textbf{Proof of symmetry results} }
 Let us now consider the system \eqref{modprob}:
 $$
\begin{cases} - \gD U =F(|x|,U) \quad &\text{in } \gO \\
U= 0 \quad & \text{on } \de  \gO
\end{cases}
$$
where $\Omega $ is a ball or an annulus in $\RN$,  $F=(f_1 , \dots , f_m)$ is a function belonging to $C^{1, \ga }([0,+\infty )\times \Rm; \Rm)$ and $U=(u_1, \dots ,u_m)$ is a vector valued function in $\gO$,  $m, N \geq 2$.\\ 
In Section 1 we defined the Morse index of a solution through the quadratic form  $Q_U$ defined in \eqref{quadraticformlineariz} associated to the linearized operator at a solution $U$, i.e. to the linear operator
\be \label{linearizedoperator} L_U (V) = - \gD V - J_F (|x|, U) V 
\ee 
 As remarked, it coincides with the  quadratic form corresponding to  the  selfadjoint operator 
 \be \label{symmetriclinearizedoperator} \tilde {L}_U (V) = - \gD V - \frac 12 \left (J_F (|x|, U) + J_F ^{\text{t}} (|x|, U)  \right )  V 
\ee 
where $J_F ^{\text{t}}$ is the transpose of the matrix $J_F $.\par 
 Hence  if $\gl _k $ and $W^k$ denote the \emph{symmetric} eigenvalues and eigenfunctions of $L_U$, i.e. $W^k$ satisfy  
 $$ 
  \begin{cases} - \gD W^k + C  W^k =   \gl _k  W^k &  \text{ in }  \gO \\
 W^k =0 & \text{ on } \de \gO \; , 
 \end{cases}
$$
where $C= c_{ij}(x)$,  $c_{ij}(x)= \frac 12 \left [ \frac {\de f_i}{\de u_j}(|x|,U(x)) + \frac {\de f_j}{\de u_i}(|x|,U(x))\right ]$ ,  \\
 as in the scalar case it is easy to prove (see \cite {DamPac}) the following

\begin{proposition} Let $\gO $ be a bounded domain in $\RN$. Then the Morse index of a solution $U$ to \eqref{modprob} equals the  number of negative \emph{symmetric} eigenvalues of the linearized operator $L_U$.
\end{proposition}

 \medskip
\subsection{Preliminary results}
Let $e\in S^{N-1}$ be a direction, i.e. $e \in \RN $, $|e|=1$,  and let us define the hyperplane $T(e)$ and the  ''cap'' $\gO (e) $ as
$$ T(e)  = \{ x \in \RN : x \cdot e=0\}\; , \quad  \gO (e)= \{ x \in \gO : x \cdot e >0\}  $$
 Moreover if $x \in \gO $ let us denote by $\gs _{e}(x)$ the reflection of $x$ through the hyperplane $T(e)$ and by $U^{\gs _{e}}$ the function $U \circ \gs _{e} $ .\\
 
  \begin{lemma}\label{lemma1} 
  \begin{itemize}
\item  Assume that $U$ is a solution of \eqref{modprob} and that the system is fully coupled along $U$ in $\gO $.
  Let us define  for any direction $e\in S^{N-1}$ the matrix $\; B^{e}(x) =\left (b_{ij}^{e}(x) \right )_{i,j=1}^m $, where
\be\nonumber  
 b_{ij}^{e}(x) = - \int _0^1 \frac {\de f _i} {\de u_j} \left [|x|, tU(x)+ (1-t)U^{\gs _e}(x)  \right ] \, dt  
\ee 
   Then  the function   $W^{e}=U-U^{\gs _{e}} = (w_1, \dots , w_m)$  satisfies  (in $\gO$ and) in  $\gO (e)$ the linear system 
\be \label{EquazDifferenza}
\begin{cases}  - \gD W + B^{e}(x) W & =0 \quad \text{ in }  \gO (e) \\
W& =0 \quad \text{ on }  \de \gO (e)
\end{cases}
\ee 
 which  is fully coupled in $\gO $ and $\gO (e)$ for any  $e\in S^{N-1}$.\\
\item If also hypothesis ii)  of Theorem \ref{f'convessa} holds, and we define for any direction $e\in S^{N-1}$
  $\quad  B^{e,s}(x) =\left (b_{ij}^{e,s}(x) \right )_{i,j=1}^m  \quad $, where
\be \label{DefCoeffSimm} b_{ij}^{e,s} (x) =  - \frac 12 \left(  \frac {\de f _i} {\de u_j} (|x|,U(x)) +  \frac {\de f _i} {\de u_j} (|x|,U^{\gs _e}(x)) \right ) 
\ee
then the linear system with matrix $B^{e,s} $
is fully coupled as well in $\gO $ and $\gO (e)$ for any  $e\in S^{N-1}$. \\
 Moreover for any  $i,j=1,\dots ,m$  and $ x \in \gO $ \\
\be \label{CfrCoefficienti} b_{ij}^{e} (x)\;  \geq \; b_{ij}^{e,s} (x)  
\ee
and the inequality is strict  for any $ i_0$, $ j_0$ such that
  $\frac {\de f _{i_0}} {\de u_{j_0}}$ satisfies the strict convexity assumption  \eqref{strictconvexity} if  $u_k(x) \neq u^{\gs _{e}}_k(x)  $
 for any $k \in \{ 1, \dots ,m  \}$. \\
 As a consequence for the  quadratic forms $ Q^{e} $ and $ Q^{e,s} $  associated to the matrixes $ B^{e} $ and $ B^{e,s} $ we have that
\begin{multline} \label {cfrformequadratiche} 0= Q^{e} (W^{e}; \gO (e)) = \int _{\gO (e)} \left [ |\nabla W^{e} |^2 + B^{e}(W^{e} ,W^{e} )\right ] dx  \\
\geq  \int _{\gO (e)} \left [ |\nabla W^{e} |^2 + B^{e,s}(W^{e} ,W^{e} )\right ] dx  = Q^{e,s}(W^{e}; \gO (e)) 
\end{multline}
with strict inequality if $F$ satisfies the hypothesis b) of Theorem \ref{corollario1}  and $W^{e}_k \neq 0$ for any $k \in \{ 1, \dots ,m  \}$.
\end{itemize} 

\end{lemma}
\begin{proof} From the equation $- \gD U = F (|x|,U(x)) $ we deduce that the reflected function $U^{\gs _{e}} $ satisfies the equation
$- \gD U^{\gs _{e}}  = F (|x|,U^{\gs _{e}}(x) ) $ and hence for the difference $W^{e}=U-U^{\gs _{e}}= \left ( u_1- u^{\gs _e}_1, \dots , u_m- u^{\gs _e}_m   \right ) $  we have \ \ \ \ $- \gD W^{e}  = F(|x|, U)-F(|x|, U^{\gs _{e}})$ \ \ \ so that 
$$- \gD w_i  =  f_i(|x|, U)-f_i(|x|, U^{\gs _{e}})=$$
$$ \quad  \sum _{j=1}^m \int _0^1  \frac {\de f _i} {\de u_j} \left [ |x|, tU(x)+ (1-t)U^{\gs _e}(x)  \right ]  dt   (u_j- u^{\gs _e}_j )  $$ 
and \eqref{EquazDifferenza} follows.\\
Since the system \eqref{modprob} is weakly coupled,  $ \frac {\de f _i} {\de u_j} \geq 0  $ for any $i \neq j$ , so that $b_{ij}^{e} \leq 0 $ and the system \eqref{EquazDifferenza} is weakly coupled as well.\\
To see that the system \eqref{EquazDifferenza} is also fully coupled let us show that if $i_0, j_0 \in \{ 1,\dots ,m \}$ are such that 
$ \frac {\de f _{i_0}} {\de u_{j_0}}(|\overline{x}|,U(\overline{x})) >0$ for some $\overline{x} \in \gO$, then $b_{i_0 j_0}^{e}(\overline{x}) <0 $.
This follows immediately by the nonnegativity and the continuity of $ \frac {\de f _{i_0}} {\de u_{j_0}}  $, using the definition of $b_{i_0 j_0}^{e}$.
This implies that \eqref{EquazDifferenza} is fully coupled in $\gO $  and, since $B$ is symmetric with respect to the reflection $\gs _e $, it is fully coupled in $\gO (e)$ as well.\\
Moreover if  hypotheses ii)  of Theorem \ref{f'convessa} holds then 
\begin{multline} - b_{ij}^{e}(x) = \int _0^1   \frac {\de f _i} {\de u_j} \left [|x|, tU(x)+ (1-t)U^{\gs _e}(x)  \right ] dt \\
\leq 
\int _0^1 \left (   t\, \frac {\de f _i} {\de u_j} \left [ |x|, U(x) \right ] + (1-t) \frac {\de f _i} {\de u_j} \left [ |x|,U^{\gs _e}(x)  \right ]  \right )dt \\
=
 \frac 12 \left(  \frac {\de f _i} {\de u_j} (|x|,U(x)) +  \frac {\de f _i} {\de u_j} (|x|,U^{\gs _e}(x)) \right )= - b_{ij}^{e,s}(x) 
 \end{multline}
 This implies \eqref{CfrCoefficienti}  and the inequality is strict  for any $ i_0$, $ j_0$ such that
  $\frac {\de f _{i_0}} {\de u_{j_0}}$ satisfies the strict convexity assumption  \eqref{strictconvexity} if  $u_k(x) \neq u^{\gs _{e}}_k(x)  $
 for any $k \in \{ 1, \dots ,m  \}$. \\
This in turn implies the full coupling of the system with matrix $B^{e,s} $ and \eqref{cfrformequadratiche} if hypothesis b) of Theorem \ref{corollario1} holds.
 \end{proof} 
 
 Next, we state the following two lemmas, whose proofs can be found in \cite{DamPac}.
\begin{lemma}\label{lemma2} Let $U=(u_1, \dots , u_m)$ be a solution of \eqref{modprob} and assume that the hypothesis i) of Theorem \ref{f'convessa} holds.
 If for every $e\in S^{N-1}$ we have either $U \geq U^{\gs _{e}}$ in $\gO (e)$ or $U \leq U^{\gs _{e}}$ in $\gO (e)$, then $U$ is foliated Schwarz symmetric.
\end{lemma}
 
 \begin{lemma}\label{lemma3} Let $U= (u_1, \dots , u_m)$  be a solution of \eqref{modprob} and assume that the hypothesis i) of Theorem \ref{f'convessa} holds.\\
Suppose that there exists a direction $e$ such that $U$ is symmetric with respect to $T(e)$ and 
the principal eigenvalue $\tilde {\gl } _1 (\gO (e)) $ of the linearized operator $L_U (V) = - \gD V - J_F (x, U) V $ in $\gO (e)$ is nonnegative.
Then $U$ is foliated Schwarz symmetric. 
\end{lemma}

\begin{lemma}\label{lemma4}  Suppose that $U$ is a solution of \eqref{modprob} with Morse index $\mu (U) \leq N-1$ and assume that the hypothesis i) of Theorem \ref{f'convessa} holds. Let  $Q^{e,s}$ be 
the quadratic form associated to the operator $L^{e,s} (V) = - \gD V + B^{e,s}V $ \ $B^{e,s}$ being defined in \eqref{DefCoeffSimm} :
\be \label{formaquadraticasimmetricalinearizzato}
\begin{split}
&Q^{e,s} (\Psi ; \gO ' ) = \int _{\gO '} \left [ |\nabla \Psi |^2 + B^{e,s}(\Psi ,\Psi )\right ] dx= \\
& \int _{\gO}\left [   \sum _{i=1}^m |\nabla \psi _i |^2 -\sum _{i,j=1}^m  \frac 12 \left ( \frac {\de f_i}{\de u_j}(|x|, U(x))+  \frac {\de f_i}{\de u_j}(|x|, U^{\gs _{e}}(x)) \right ) \psi _i \psi _j  \right ]  \, dx 
\end{split}
\ee
Then there exists a direction $e \in S^{N-1}$ such that 
$$
 Q^{e,s} (\Psi ; \gO (e) )  \geq 0
  \notag
$$
for any $\Psi \in C_c^1 (\gO (e);\Rm)$. Equivalently  
the first \emph{symmetric} eigenvalue $\gl _1^{\text{s}} (L^{e,s}, \gO (e))$ of the  operator $L^{e,s} (V) = - \gD V + B^{e,s}V $ in $\gO (e)$ is nonnegative (and hence also the principal eigenvalue $\tilde {\gl }_1 (L^{e,s} ,\gO (e) )$ is  nonnegative).
\end{lemma}

\begin{proof} 
Let us assume that  $1\leq j= \mu (U) \leq N-1 $ and let $\Phi _1, \dots , \Phi _j$ be mutually orthogonal eigenfunctions corresponding to the negative symmetric eigenvalues $\gl _1^{\text{s}} (L_U, \gO )$, \dots , $\gl _j^{\text{s}} (L_U, \gO )$ of the linearized operator $L_U (V) = - \gD V - J_F (x, U) V $ in $\gO $ .\\
 For any $e\in S^{N-1}$ let $\phi ^{e,s}$ be the first positive $L^2$ normalized eigenfunction  of the symmetric system associated to the linear operator 
 $L^{e,s}  $ in $\gO (e)$. We observe that $\phi ^{e,s}$ is uniquely determined since the corresponding system is fully coupled in $\gO (e)$.
 Let $\Phi ^{e,s}$ be the odd extension of $\phi ^{e,s}$ to $\gO$, and let us observe that  $\Phi ^{-e,s} = - \Phi ^{e,s}$, because $B^{e,s}$ is symmetric with respect to the reflection $\gs _e$. \\
 The mapping $e \mapsto \Phi ^{e,s} $ is a continuous odd function from $S^{N-1}$ to $\huno $,    
  therefore the mapping $h: S^{N-1}\to \R ^{j} $ defined by
 $$ h(e)= \left ( (\Phi ^{e,s} \, , \, \Phi _1 )_{\ldue (\gO )},   \dots , ( \Phi ^{e,s} \,, \, \Phi _j )_{\ldue (\gO )} \right )
 $$
 is an odd continuous mapping, and since $j \leq  N-1 $,  by the Borsuk-Ulam Theorem it must  have a zero.
 This means that there exists a direction $e\in S^{N-1}$ such that
   $\Phi ^{e,s}  $ is orthogonal to all the eigenfunctions $\Phi _1, \dots , \Phi _j$.   This implies that 
 $Q_U(\Phi ^{e,s}; \gO )  \geq 0 $, because $ \mu (U)=j $, and since $\Phi ^{e,s}$ is an odd function, we obtain that 
 $0 \leq Q_U(\Phi ^{e,s} ; \gO )  =  Q^{e,s} (\Phi ^{e,s}, \gO )= 2 Q^{e,s} (\phi ^{e,s}, \gO (e)) = 2  \gl _1^{\text{s}} (L^{e,s}, \gO (e)) $
 \end{proof}

\subsection{Proof of Theorems \ref{f'convessa} and Theorem \ref{corollario1}}
\begin{proof}[Proof of Theorem \ref {f'convessa}] 
By Lemma \ref{lemma4} there exists a direction $e$ such that 
the first symmetric eigenvalue $\gl _1^{\text{s}} (L^{e,s}, \gO (e))$ of the  operator $L^{e,s} (V) = - \gD V + B^{e,s}V $ in $\gO (e)$ is nonnegative, and hence also the principal eigenvalue $\tilde {\gl }_1 (L^{e,s} ,\gO (e) )$ is  nonnegative.\par 
By \eqref{CfrCoefficienti} and Proposition \ref{varformautov} ix), the first symmetric eigenvalue $\gl _1^{\text{s}} (- \gD +B^{e}, \gO (e))$ of the operator 
$- \gD +B^{e}$ in $ \gO (e)$, $B^{e}$ being defined in \eqref{DefCoeffSimm}, is also nonnegative.\\
If $\gl _1^{\text{s}} (- \gD +B^{e}, \gO (e)) >0$, then necessarily the difference $W^{e}=U-U^{\gs _{e}} $ must vanish. 
In fact, since it satisfies the equation 
\eqref{EquazDifferenza}, we get for the associated quadratic form
$$
Q^{e}  (W^{e} ; \gO (e)) = \int _{\gO (e)} \left [ |\nabla W^{e} |^2 + B^{e}(W^{e} ,W^{e} )\right ] dx =0 
$$
so by Proposition \ref{varformautov} $W^{e} =0$, and $U \equiv U^{\gs _{e}}$. 
 This implies that $B^{e}= B^{e,s}=J_{F}(|x|, U)$, so that we find a direction $e$ satisfying the hypotheses of Lemma \ref{lemma3}, and we get that  $U$ is  
 foliated Schwarz symmetric.\\
If instead $\gl _1^{\text{s}} (-\Delta +B^e, \gO (e))=0$, then necessarily $\gl _1^{\text{s}} (L^{e,s}, \gO (e))= \gl _1^{\text{s}} (- \gD +B^{e}, \gO (e))= 0$.\\
 Let us now remark for future use in the proof of Theorem \ref {corollario1} that if hypothesis b) of Theorem \ref{corollario1} holds then necessarily $W^{e}=0$, so that even in this case we find a direction 
 $e$ such that 
 $U$ is symmetric with respect to $T(e)$ and not only the principal eigenvalue $ \tilde {\gl} _1 (L_U, \gO (e'))$ of the linearized operator  in $\gO (e')$ is nonnegative, but also the first symmetric eigenvalue $ \gl _1 ^{\text{s}} (L_U,\gO (e)) = \gl _1 ^{\text{s}} (L_U,\gO (-e))  \geq 0 $.\\
In fact if 
$W^{e}$ would not vanish, it should be the first eigenfunction of the system \eqref{EquazDifferenza}, which by  Lemma \ref{lemma1}  is fully coupled. So it would be strictly positive (or negative) and by Lemma \ref{lemma1} and the hypothesis b) we would get
\begin{multline} 0= Q^{e} (W^{e}; \gO (e)) = \int _{\gO (e)} \left [ |\nabla W^{e} |^2 + B^{e}(W^{e} ,W^{e} )\right ] dx  \\
> \int _{\gO (e)} \left [ |\nabla W^{e} |^2 + B^{e,s}(W^{e} ,W^{e} )\right ] dx  = Q^{e,s}(W^{e}; \gO (e)) 
\end{multline}
 contradicting the nonnegativity of the first symmetric eigenvalue $\gl _1^{\text{s}} (L^{e,s}, \gO (e))$ of the  operator $L^{e,s} (V) = - \gD V + B^{e,s}V $ in $\gO (e)$.\\
Then the only case left is  when \ $\gl _1^{\text{s}} (-\Delta +B^e, \gO (e))=0$
,   hypothesis b) does not hold,  and $W^{e}$ does not vanish, so that it must be the first symmetric eigenfunction of the system  \eqref{EquazDifferenza}, which is fully coupled.  \\
 This implies that   it does not change sign in $\Omega (e)$ , and assuming that e.g. $U \geq U^{\gs _{e}} $ then, by the strong maximum principle   we have that $U > U^{\gs _{e}} $  in $\gO (e)$. \\
 We now apply, as in \cite{PW}, \cite{GPW} and \cite {DamPac}, the ''rotating plane method'', which is an adaptation of the moving plane method as developed in \cite {BN} and obtain a different direction $e'$ such that $U$ is symmetric with respect to $T(e')$ and 
the principal eigenvalue $ \tilde {\gl} _1 (L_U, \gO (e'))$ of the linearized operator  in $\gO (e')$ is nonnegative.
Then by Lemma \ref{lemma3} we will get that  $U$ is foliated Schwarz symmetric. \\
More precisely, without loss of generality we suppose that $e=(0,0,\dots ,1)$ and  for $\theta \geq 0 $ we set $e_{\theta}=(\sin \theta ,0,\dots ,\cos \theta) $, so that $e_0=e$, and
$\gO _{\theta}= \gO (e_{\theta})$, $U^{\theta}= U^{\gs _{e_{\theta}}}$, $W^{\theta}=U-U^{\gs _{e_{\theta}}} $, 
Let us define $\theta _0 = \sup \{ \theta \in [0, \pi ) : U > U ^{\theta} \text{ in } \gO _{\theta}  \}$.
Then necessarily $\theta _0 < \pi$, since $(U-U^0)(x) = - (U-U^{\pi})(\gs _{e_{\pi}}(x))$ for any $x \in \gO _0$ (and $\gs _{e_{\pi}}(x)) \in \gO _{\pi}$).\\
Suppose by contradiction that $U \not \equiv U^{\theta _0}$ in $\gO _{\theta _0}$.
Then, by Theorem  \ref{SMP} applied to    the difference $W^{\theta _0}=(U - U^{\theta _0})$,  we get that $W^{\theta _0}>0$ in $\gO _{\theta _0}$. Taking a compact $K \subset \gO _{\theta _0} $ whith small measure and such that (componentwise) 
$W^{\theta _0}>(\eta, \dots , \eta)$ for some $\eta >0$, for $\theta $ close to $\theta _0 $ we still have that $W^{\theta}>(\frac {\eta }2, \dots , \frac {\eta }2) $ in $K$, while 
$W^{\theta}>0$ in $\gO _{\theta} \setminus K$ by the weak maximum principle in domains with small measure. 
This implies that for $\theta $ greater than and close to $\theta _0 $
the inequality $U > U ^{\theta} \text{ in } \gO _{\theta} $ still holds, contradicting the definition of $\theta _0$. 
Therefore $U \equiv U^{\theta _0}$ in $\gO _{\theta _0}$.\\
Observe that the difference $ W^{\theta} $ satisfies the linear system \eqref{EquazDifferenza} and does not change sign for any $\theta \in [0, \theta _0)$, which implies by Proposition \ref{principaleigenvalue} that it is the principal eigenfunction for the system \eqref{EquazDifferenza} corresponding to the eigenvalue $\tilde {\gl} _1 =0 $.
As $\theta \to \theta _0 $ the system \eqref{EquazDifferenza} tends to the linearized system, in the sense that the coefficients $-b_{ij}(x)$ tend to the derivatives 
$\frac {\de f_i}{\de u_j}(|x|, U(x))$.  Then by continuity  the principal eigenvalue $\tilde {\gl } _1 (L_U, \gO (e_{\theta _0})) $ of the linearized operator $L_U (V) = - \gD V - J_F (x, U) V $ in $\gO (e_{\theta _0})$ is zero and the proof of Theorem \ref{f'convessa} is complete. 
\end{proof}
\begin{proof}[Proof of Theorem \ref {corollario1}] 
 To prove Theorem \ref{corollario1}, let us recall that, as observed above, if hypothesis b) holds then we can find a direction $e$ such that 
 $U$ is symmetric with respect to $T(e)$ and not only the principal eigenvalue $ \tilde {\gl} _1 (L_U, \gO (e))$ of the linearized operator  in $\gO (e)$ is nonnegative, but also the first symmetric eigenvalue $ \gl _1 ^{\text{s}} (L_U,\gO (e)) = \gl _1 ^{\text{s}} (L_U,\gO (-e))  \geq 0 $.\\
 The same happens if  $U$ is a Morse index one solution. 
 In fact in this case for any direction $e\in S^{N-1}$  at least one amongst $\gl _1^{\text{s}} (L_U, \gO (e))$  and 
$\gl _1^{\text{s}} (L_U, \gO (-e))$ must be nonnegative, otherwise taking the  corresponding first eigenfunctions we would obtain a $2$-dimensional subspace of  $C_c^1(\gO ; \Rm)$ where the quadratic form is negative definite, so in the symmetry direction $e$ found above we have that 
$ \gl _1 ^{\text{s}} (L_U,\gO (e)) = \gl _1 ^{\text{s}} (L_U,\gO (-e))  \geq 0 $.\par
After a rotation, we may assume that $e=e_2=(0,1,\dots,0)$ so that $T(e)\,=\,\{x\in \R^N\,:\,x_2=0\}$
and we may introduce new (cylinder) coordinates  
$(r, \theta, y_3, \dots ,y_N)$ defined by the relations $x= r [\cos \theta e_1 + \sin \theta e_2] + \sum_{i=3}^N y_i e_i $.   \\ 
Then the angular derivative $U_{\theta} $ of $U$ with respect to $\theta $, extended by zero at  the origin if $\gO $ is a ball, satisfies the linearized system, i.e.
\be \label{sistemalinearizzato}
- \gD U _{\theta} - J_F (|x|, U) U_{\theta} =0 \quad \text{ in } \gO(e_2).  
\ee
Moreover, by the symmetry of $U$ with respect to the hyperplane $T(e_2)$, we have that $U_{\theta}$ is antisymmetric with respect to $T(e_2)$ and therefore vanishes on $T(e_2)$ and since it  vanishes
on $\de \gO$, it vanishes on $\de \gO(e_2)$ as well.\\
If $ \tilde {\gl} _1 (L_U, \gO (e_2)) >0 $ then, by the maximum principle, the derivatives $U_{\theta} $, must vanish, and hence 
$U$ is radial. So if  $U_{\theta} \not \equiv 0 $ necessarily, from \eqref{sistemalinearizzato}, $ \tilde {\gl} _1 (L_U, \gO (e_2)) = \gl _1 ^{\text{s}} (L_U,\gO (e_2))  =0 $ and by v) of Proposition \ref{principaleigenvalue} 
$U_{\theta}$ is the first eigenfunction of the simmetrized system, as well as a solution of \eqref{sistemalinearizzato} (let us observe that by the symmetry of $U$ the linearized system \eqref{sistemalinearizzato} is fully coupled in $\gO(e_2)$).
So  we get that 
$J_F (|x|, U) U_{\theta}= \frac 12 \left (J_F (|x|, U) + J_F ^t (|x|, U) \right ) U_{\theta}$, i.e. 
  \eqref{superfullycoupling1} and if $m=2$, since $U_{\theta}$ is positive, we get    \eqref{superfullycoupling2}.
\end{proof}

 \section{  \textbf{Some examples} }
 A first type of elliptic systems that could be considered are those of ''gradient type'' (see \cite{deF2}), i.e. systems of the type \eqref{modprob}
 where $f_j(|x|,U)= \frac {\de g} {\de u_j}(|x|,U)$ for some function $g \in C^{2, \ga }([0,+\infty ) \times \Rm)$.\\
 In this case the solutions correspond to critical points of the functional
 $$ \Phi (u) = \frac 12 \int _{\gO } |\nabla U|^2 \, dx - \int _{\gO } g(|x|,U) \, dx
 $$
 in $\huno (\gO )$ and the linearized operator \eqref{linearizedoperator} coincides with the second derivative of $\Phi $.\par 
 Thus standard variational methods apply which often give solutions of finite (linearized) Morse index, as, for example, in the case  when the Mountain Pass Theorem can be used or when one considers the so called ''least energy'' nodal solutions. So, if the hypotheses of Theorem \ref{f'convessa} are satisfied, our symmetry results can be applied (see also \cite{DamPac}).\par
 \smallskip
 A second type of interesting systems  of two equations are the so called ''Hamiltonian type'' systems (see  \cite{deF2} and the references therein).
 More precisely we consider the system
 \be \label{hamiltoniansystem}
\begin{cases}  
 - \gD u_1  = f_1(|x|,u_1, u_2) & \text{ in } \gO  \\
 - \gD u_2 = f_2(|x|,u_1, u_2)   & \text{ in } \gO \\
 u_1=  u_2  =0 & \text { on } \de \gO 
\end{cases}
\ee
  with
  \be \label{hamiltonianconditions}
  f_1 (|x|,u_1, u_2)=  \frac {\de H} {\de u_2}(|x|,u_1,u_2) \quad , \quad f_2 (|x|,u_1, u_2)=  \frac {\de H} {\de u_1}(|x|,u_1,u_2)
  \ee
  for some scalar function $H \in C^{2, \ga }([0,+\infty ) \times \R ^2)$.
  These systems  can be studied by considering the associated functional
   \be J(U)=I(u_1, u_2)= \frac 12 \int _{\gO }  \nabla u_1 \cdot \nabla u_2  \, dx - \int _{\gO } H(|x|,u_1, u_2) \, dx
 \ee
 either in  $ \huno (\gO) $ or in other suitable Sobolev spaces (see \cite{deF2}). \par
 It is easy to see that the linearized operator defined in \eqref{linearizedoperator} does not correspond to the second derivative of the functional $J$, which is strongly indefinite. Nevertheless  solutions of \eqref{hamiltoniansystem} can have finite linearized Morse index  as shown in \cite {DamPac}. \par
 \smallskip
 In particular we consider the  system  
 \be \label{potenza}
 \begin{cases}
 - \Delta u_1 &= |u_2|^{p-1 } u_2  \text{ in }  \gO  \\
- \Delta u_2 &=  |u_1|^{q-1 } u_1  \text{ in }  \gO  \\
  u_1=u_2&=0 \text{ on } \de \gO 
 \end{cases}
 \ee
 where $1< p, q < \frac {N+2} {N-2}$.
 Then we  start from the case $p=q$ and the solution $u_1=u_2=z$, where $z$ is a scalar  solution of the equation
 \be \label{equazionescalare} 
  \begin{cases}
 - \Delta z &= |z|^{p-1 }z  \text{ in }  \gO  \\
 z&=0 \text{ on } \de \gO 
 \end{cases}
 \ee
 
 Let us observe that if $p=q$  and $z$  has Morse index equal to the integer $\mu (z)$, then $\mu (z)$ is also the Morse index  of the solution 
 $U=(u_1,u_2)=(z,z)$ of the system \eqref{potenza}. Indeed the linearized equation at $z$ for the equation \eqref{equazionescalare} and the linearized system at $(z,z)$ for the system \eqref{potenza} are respectively
  \be \label{linearizzatoequazionescalare} 
  \begin{cases}
 - \Delta \phi  - p |z|^{p-1 } \phi & =0  \text{ in }  \gO  \\
  \phi &=0 \text{ on } \de \gO 
 \end{cases}
 \ee
and
 \be \label{linearizzatopotenza}
 \begin{cases}
  - \Delta \phi _1  - p |z|^{p-1 } \phi _2 & =0  \text{ in }  \gO  \text{ in }  \gO  \\
 - \Delta \phi _2  - p |z|^{p-1 } \phi _1 & =0  \text{ in }  \gO   \text{ in }  \gO  \\
  \phi _1=\phi _2 &=0 \text{ on } \de \gO 
 \end{cases}
 \ee
This implies that the eigenvalues of these two operators are the same, since if $\phi $ is an eigenfunction for \eqref{linearizzatoequazionescalare} 
corresponding to the eigenvalue $\gl _k $ then taking 
$\phi _1 = \phi _2 = \phi $ we obtain an eigenfunction $(\phi _1 , \phi _2)$ for \eqref{linearizzatopotenza} corresponding to the same eigenvalue, while if 
 $(\phi _1 , \phi _2)$  is an eigenfunction for  \eqref{linearizzatopotenza} corresponding to the eigenvalue $\gl _k $ then $\phi = \phi _1 + \phi _2 $ is an 
 eigenfunction for \eqref{linearizzatoequazionescalare} corresponding to the same eigenvalue.\par 
 So if we start from a nondegenerate solution of \eqref{equazionescalare} with a fixed exponent 
 $\overline {p} \in (1, \frac {N+2} {N-2} )$, using the Implicit Function Theorem,  we find a branch of  solutions of \eqref{potenza} corresponding to (possibly different) exponents $p,q$ close to $\overline {p} $.
 For example if we start with a least energy nodal solution $z$ in the ball of equation \eqref{equazionescalare} with the exponent $\overline {p} $,  knowing that its Morse index is two  we get a branch of Morse index two solutions for  $p,q $ close to $\overline {p} $.\par 
 Note that, as proved in \cite{AP}, the least energy nodal solution of \eqref{equazionescalare} is not radial but foliated Schwarz symmetric. So it is obviously degenerate, but working in the space of axially symmetric functions 
  we could remove the degeneracy and apply the continuation method described above, if there are no other degeneracies. \par
Thus, starting from an exponent $\overline{p}$ for which the least energy nodal solution of \eqref{equazionescalare}  is not degenerate, we can construct solutions 
$U=(u_1, u_2)$ of \eqref{potenza} in correspondence of exponents $p,q$ close to $\overline{p}$, with Morse index two.
Then Theorem \ref{f'convessa} applies if $p,q \geq 2$ and, in particular, we get that the coupling condition \eqref{superfullycoupling2} holds, which in this case can be written as
  \be \label{superfullycoupling3}
  p|u_2|^{p-1} =  q|u_1|^{q-1}   \quad \text{in } \gO 
\ee 
 Note that more generally, by Theorem \ref{corollario1}, the equality \eqref{superfullycoupling3} must hold for every solution $U=(u_1,u_2)$ of \eqref{potenza} with Morse index $\mu (U) \leq N-1$,  giving  so a sharp condition to be satisfied by the components of a solution of this type.\par
 Let us remark that our results apply also when the nonlinearity depends on $|x|$ (in any way).
 Arguing as before it is not difficult to construct systems having solutions with low Morse index, in particular with Morse index one or two. \par
  An example could be the "Henon system"
  \be \label{Henon}
 \begin{cases}
 - \Delta u_1 = |x|^{\alpha }|u_2|^{p-1 } u_2 & \text{ in }  \gO  \\
 - \Delta u_2 =  |x|^{\beta } |u_1|^{q-1} u_1 &\text{ in }  \gO  \\
 u_1, u_2 >0  &\text{ in } \gO \\
  u_1=u_=0 & \text{ on } \de \gO 
 \end{cases}
 \ee
 with $\alpha , \beta >0$, $ p,q \geq 2 $. \par

\end{document}